\documentclass[12pt,leqno]{amsart}
\usepackage{amsmath, amssymb, latexsym, amscd, amsthm,amsfonts,amstext}
\usepackage[mathscr]{eucal}

\theoremstyle{plain}
\newtheorem{lemma}{Lemma}[section]
\newtheorem{theorem}[lemma]{Theorem}

\newtheorem{proposition}[lemma]{Proposition}

\newtheorem{remark}[lemma]{Remark}

\def\underset#1#2{{\mathrel{\mathop {{}_{} {#2}}\limits_{{#1}_{}}}}}
\def\upplim_#1{\underset{#1}{\overline\lim}\;}
\def\lowlim_#1{\underset{#1}{\underline\lim}\;}
\newcommand{\ord}{{\mathrm{ord}}}

\numberwithin{equation}{section}

\begin{document}
\title[ ]{An effective schmidt's subspace theorem for hypersurfaces in subgeneral position in projective varieties over function fields}
\author{Giang Le}
\maketitle
\begin{center}
{\it Department of Mathematics, Hanoi National University of Education,}

{\it 136-Xuan Thuy, Cau Giay, HaNoi, VietNam.}

\textit{E-mail: legiang01@yahoo.com}
\end{center}
\begin{abstract} {We deduce an effective version of Schmidt's subspace theorem on a smooth projective variety $X$ over function fields of characteristic zero for hypersurfaces located in $N-$subgeneral position with respect to $X$. }
\end{abstract}

\def\thefootnote{\empty}
\section{Introduction}
Schmidt's subspace theorem is one of the most important results in the developments of Diophantine approximation. In the number field case, there is still no effective version of this theorem. On the other hand, with techniques from Nevanlinna theory it has become possible to obtain effective version of several important results in Diophantine approximation over algebraic function fields. In \cite{AW}, An and Wang obtained an effective Schmidt's subspace theorem for non-linear forms over function fields. In \cite{RW}, Ru and Wang extended such effective results to divisors of a projective variety $X\subset\mathbb{P}^M$ over function fields of characteristic 0 coming from hypersurfaces in $\mathbb{P}^M.$ Our purpose is to generalize the above results to the case in which hypersurfaces are located in $N-$subgeneral position with respect to $X$.

Recently, Chen, Ru, Yan (see \cite{CRY2}) and Levin (see \cite{Levin}, Theorem 5.1) established Schmidt's subspace theorem for hypersurfaces located in N-subgeneral position over number fields and showed the analogous result for the case of holomorphic curves. This paper is inspired by these works.

 The method given in this paper, when applied to the number field case, actually simplifies the proof in \cite{CRY2} for the case of smooth projective varieties. However, in the case of function fields we need to make explicit all the constants involved in order to obtain an effective version. For this, we need to make the constructions in \cite{CRY2} more explicit and we also need the effective version of the classical Schmidt's subspace theorem for linear forms as in \cite{W}, an effective upper bound and lower bound for Hilbert functions and an effective version of Hilbert's Nullstellensatz. In section 2, we will describe a canonical way to find polynomials from the Chow form of $X$ and generate an ideal whose radical is $I_X$. We will give some effective results in Section 3 and the main result will be deduced in Section 4.

To state our results, we will recall some definitions and basic facts from algebraic geometry.

Let $k$ be an algebraically closed field of characteristic 0 and let $V$ be  a nonsingular projective variety. We will fix an embedding of  $V$ into a projective space $\mathbb{P}^{M_0}$. Denote by $K=k(V)$ the function field of $V$. Let $M_K$ denote the set of prime divisors of $V$ (irreducible subvarieties of codimension one). Let $\mathfrak{p}$ be a prime divisor. As \emph{V} is nonsingular, the local ring $\mathcal{O}_\mathfrak{p}$ at $\mathfrak{p}$ is a discrete valuation ring. For each $x\in K^*$, its order $\ord_\mathfrak{p}x$ at $\mathfrak{p}$ is well defined. We can associate to \emph{x} its divisors
$$(x)=\sum_{\mathfrak{p}\in M_K}\ord_\mathfrak{p}(x)\mathfrak{p}=(x)_0-(x)_{\infty},$$
where $(x)_0$ is the zero divisor of $x$ and $(x)_{\infty}$ is the polar divisor of $x$ respectively. Let deg $\mathfrak{p}$ denote the projective degree of $\mathfrak{p}$ in $\mathbb{P}^{M_0}$. Then the sum formula
$$\text{deg}(x)=\sum_{\mathfrak{p}\in M_K}\ord_\mathfrak{p}(x)\deg\,\mathfrak{p}=0$$ 
holds for all $x\in K^*$.

Let $x=[x_0:x_1:\cdots:x_M]\in\mathbb{P}^M(K)$ and define
$$e_{\mathfrak{p}}(x):=\min\limits_{0\leq i\leq M}\{\ord_{\mathfrak{p}}(x_i)\}.$$
The (logarithmic) height of $x$ is defined by the following formula:
$$h(x)=-\sum_{\mathfrak{p}\in M_K}e_{\mathfrak{p}}(x)\text{deg}\,\mathfrak{p}.$$
For $x\in K^*$, the logarithmic height of $x$ is defined by:
$$h(x)=-\sum_{\mathfrak{p}\in M_K}\min\{0,\ord_{\mathfrak{p}}(x)\}\text{deg}\,\mathfrak{p}.$$
By the sum formula, it is easy to see that for $x\in K^*$,
$$h(x)=\sum_{\mathfrak{p}\in M_K}\max\{0,\ord_{\mathfrak{p}}(x)\}\text{deg}\,\mathfrak{p}.$$
Note that, the definition of $e_{\mathfrak{p}}(x)$ depends on the choice of the coordinates of $x$. Apart from this, the height function is well-defined on $\mathbb{P}^M(K)$.

Let $Q$ be a homogeneous polynomial of degree $d$ in $K[X_0,\cdots,X_M],$ say $Q(X)=\sum_Ia_I X_0^{i_0}\cdots X_M^{i_M}$ where the sum is taken over all index sets $I=\{i_0,\ldots,i_M\}$ such that $i_j\geq 0$ and $\sum\limits_{j=0}^Mi_j=d$. For $\mathfrak{p}\in M_K,$ we set
$$e_{\mathfrak{p}}(Q):=\text{min}_I\{\ord_{\mathfrak{p}}(a_I)\}.$$

We define the height of a homogeneous polynomial $Q$ of degree $d$ in $K[X_0,\ldots,X_M]$ as
$$h(Q)=\sum_{\mathfrak{p}\in M_K}-e_{\mathfrak{p}}(Q)\deg \mathfrak{p}.$$
From the sum formula, we have $h(\alpha Q)=h(Q)$ if $\alpha\in K^*.$ This also shows that $h(Q)\geq 0$ since we may assume that one of the nonzero coefficient of $Q$ is $1$.

The Weil function $\lambda_{\mathfrak{p},Q}$ is defined by
$$\lambda_{\mathfrak{p},Q}(x):=\left(\ord_{\mathfrak{p}}(Q(x))-de_{\mathfrak{p}}(x)-e_{\mathfrak{p}}(Q)\right)\text{deg}\,\mathfrak{p}\geq 0$$
for $x\in\mathbb{P}^M(K)\backslash\{Q=0\}.$

Let $Q_1, Q_2,\ldots, Q_m$ be homogeneous polynomials of degree $d$ in $K[X_0,\ldots,X_M]$. Recall that, we have
$$e_\mathfrak{p}(Q_1+\cdots+Q_m)\geq \min\{e_\mathfrak{p}(Q_1),\ldots,e_\mathfrak{p}(Q_m)\}$$
$$e_\mathfrak{p}(Q_1\cdots Q_m)=e_\mathfrak{p}(Q_1)+\cdots+e_\mathfrak{p}(Q_m).$$

We define $$e_\mathfrak{p}(Q_1,\ldots,Q_q)=\min\{e_\mathfrak{p}(Q_1),\ldots,e_\mathfrak{p}(Q_m)\}$$
$$h(Q_1,\ldots,Q_m)=-\sum_{\mathfrak{p}\in M_K}e_\mathfrak{p}(Q_1,\ldots,Q_q)\deg \mathfrak{p}.$$

Let $X$ be a $n$-dimensional projective subvariety of $\mathbb{P}^M$ defined over $K$. The height of $X$ is defined by $$h(X):=h(F_X),$$
where $F_X$ is the Chow form of $X$.

 Let $N$ be a positive integer, $N\geq n$. Homogeneous polynomials $Q_1,\ldots,Q_q\\\in K[X_0,\ldots,X_M], q\geq n+1,$ are said to be in \emph{N-subgeneral position with respect to X} if $\cap_{j=1}^{N+1}(\{Q_{i_j}=0\})\cap X(\bar{K})=\emptyset$ for any distinct $i_1,\ldots,i_{N+1}\in \{1,\ldots,q\},$ where $\bar{K}$ is the algebraic closure of $K$. When $N=n,$ they are said to be  in \emph{general position with respect to X}.

In this paper, we will prove the following effective version of the generalized Schmidt's  subspace theorem over $K$.
\begin{theorem}\label{t:11}
 Let $K$ be the function field of a nonsingular projective variety $V$ defined over an algebraically closed field of characteristic 0. Let $X$ be a smooth n-dimensional projective subvariety of $\mathbb{P}^M$ defined over K with projective degree $\triangle$. Let $Q_i, 1\leq i\leq q,$ be homogeneous polynomials of degree $d_i$ in $K[X_0,\ldots,X_M]$ in N-subgeneral position with respect to X and let S be a finite set of prime divisors of V. Then for any given $\epsilon>0$, there exists an effectively computable finite union $\mathfrak{U}_{\epsilon}$ of proper algebraic subsets of $\mathbb{P}^M(K)$ not containing $X$  and effectively computable constants $c_{\epsilon}, \tilde{c}'_{\epsilon}$ such that for any $x\in X \backslash\mathfrak{U}_{\epsilon}$ either
$$h(x)\leq c_{\epsilon},$$
or 
$$\sum_{i=1}^q\sum_{\mathfrak{p}\in S}d_i^{-1}\lambda_{\mathfrak{p},Q_i}(x)\leq (N(n+1)+\epsilon)h(x)+c'_{\epsilon}.$$

The algebraic subsets in $\mathfrak{U}_\epsilon$ and the constants $c_{\epsilon}, c'_{\epsilon}$ depend on $\epsilon, M,N,$ $q, K, S, X$ and the $Q_i.$
\end{theorem}
\begin{remark} The constants $c_\epsilon, c'_\epsilon$ will be given in \eqref{e:015} and \eqref{e:315}. They depend on $\epsilon,$ the degree of the canonical divisor class of $V$, the projective degree of $V$, the degree of $S$ (i.e $\sum_{\mathfrak{p}\in S} \deg \mathfrak{p}$), the
projective degree of $X$, the dimension of $X$, the height of $X$ and the $Q_i,$ $q$ and $N, M$  
\end{remark}

\section{Canonical polynomials from Chow forms.}
\noindent
 \textbf{2.1.} Let $X$ be a $n$-dimensional irreducible projective subvariety of  $\mathbb{P}^M$ defined over $K$ of degree $\triangle$. To $X$, we can associate, up to a constant scalar, a unique polynomial
$$F_X(u_0,\ldots,u_n)=F_X(u_{00},\ldots,u_{0M};\ldots; u_{n0},\ldots,u_{nM})$$
in $(n+1)$ blocks of variables $u_i=(u_{i0},\ldots,u_{iM}), i=0,\ldots, n,$ which is called the \emph{Chow form} of $X$, with the following properties:

$F_X$ is irreducible,

 $F_X$ is homogeneous in each block $u_i, i=0,\ldots,n,$

$F_X(u_0,\ldots,u_n)=0$ if and only if $X\cap H_{u_0}\cap\ldots\cap H_{u_n}$ contains a $\bar{K}$
-rational point,
where $H_{u_i}, i=0,\ldots, n$ are hyperplanes given by $u_i.x=u_{i0}x_0+\cdots+u_{iM}x_M=0.$ It is well-known that the degree of $F_X$ in each block $u_i$ is $\triangle$.
 \vskip0.2cm
 \noindent
\textbf{2.2.} Let $I_X$ be the homogeneous prime ideal defining X. There is a canonical way to find polynomials from the Chow form $F_X$ of $X$ which determine $X$ set theoretically.

We now recall this construction from \cite{Br} and \cite{RW}.\,
For every $0\leq j<k\leq M$, let  $A_{jk}=(a_{\alpha\beta})$ be the $(M+1)\times (M+1)$ matrix with zero entries except that $\alpha_{jk}=1$ and $\alpha_{kj}=-1$. Since a generic skew symmetric $(M+1)\times (M+1)$ matrix $S$ has the form
$$S=\sum_{j<k}s_{jk}A_{jk}$$
for indeterminants $s_{jk}$, the coefficients $u=(u_0,\ldots, u_M)$ of a generic hyperplane passing through $x$ are given by
\begin{align}
\label{e:001} u=Sx
\end{align}
Let $F_X$ be the Chow form of $X$. We note that the coefficients of $F_X$ are in $K$ since $X$ is defined over $K$. Let $S^{(i)}=(s_{jk}^{(i)}), (0\leq i\leq n)$ be generic skew symmetric $(M+1)\times (M+1)$ matrices. Let $\mathcal{M}$ be the set of all monomials in the $n+1$ blocks of variables $s^{(i)}=(s^i_{jk}:0\leq j<k\leq M+1), (0\leq i\leq n)$, which are homogeneous of degree $\triangle$ in each block.  Then by \eqref{e:001}, we may write
\begin{align}
\label{e:002} F_X(S^{(0)}x, \ldots, S^{(n)}x)=\sum_{\sigma\in\mathcal{M}} P_{\sigma}(x)\sigma.
\end{align}
Since $S^{(0)}x,\ldots,S^{(n)}x$ are generic hyperplanes through $x$, $F_X(S^{(0)}x,\ldots,S^{(n)}x)=0$ if and only if $x\in X$. On the other hand, from \eqref{e:002} we have that $F_X(S^{(0)}x,\ldots,S^{(n)}x)=0$ if and only if $P_\sigma(x)=0$ for all $\sigma\in\mathcal{M}.$
We conclude that
\emph { $I_X$ is the radical of the ideal generated by $P_\sigma, (\sigma\in\mathcal{M}).$}
We also recall the following result of Catanese \cite{Ca}
\begin{theorem} \label{theoremA}
 If $X$ is a smooth projective variety in $\mathbb{P}^M,$ then the polynomials $P_\sigma, (\sigma\in \mathcal{M})$ cut out $X$ scheme-theoretically. In other words, if $p_{\sigma,i}$ denotes the dehomogenization of $P_\sigma$ in the affine piece $X_i\not=0$ for $i=0,\ldots,n$ the ideal generated by $p_{\sigma,i}, (\sigma\in\mathcal{M})$ equals to the ideal $I_{X\cap U_i}$, where $U_i=\{U_i\not=0\}.$
 \end{theorem}

Finally, we list some information on $P_\sigma.$ First, it is clear from the construction that the degree of $P_\sigma$ is $(n+1)\triangle$. By \eqref{e:002}, the coefficients of $P_\sigma$ are $\mathbb{Z}-$linear combinations of coefficients of the Chow form $F_X$, hence
\begin{align}\label{e:003}
e_\mathfrak{p}(P_\sigma)\geq e_\mathfrak{p}(F_X)
\end{align}

 On can also verify that the number of generating polynomials $P_\sigma$ is at most
\begin{align}\label{e:004}\begin{pmatrix}
(n+1)\triangle+\frac{M(M-1)}{2}\\
(n+1)\triangle
\end{pmatrix}^{n+1}
\end{align}

\section{Some effective results} 
Let X be a projective variety of $\mathbb{P}^M$ defined over $K$.  Let $I_X$ be the prime ideal of  $K[X_0,\ldots, X_M]$ consisting of all homogeneous polynomials vanishing identically on $X$ and $H_X$ be the Hilbert function of $X$.

  We  have a lower bound  and a upper bound for the Hilbert function, due to Chardin \cite{Ch}, Nesterenko (see \cite {Nes}) or Sombra (see \cite {Som}, Theorem 4). Notice that, Chardin, Nesterenko, Sombra state their results in more general settings but we only recall their results in a special case of a function field over algebraically closed field of characteristic 0 and of an ideal of a variety.
 \begin{lemma}[Chardin \cite{Ch}]\label{l:21} Let X be a projective subvariety of $\mathbb{P}^M$ defined over K of dimension n and degree $\triangle$. Then, for $m\geq 1,$
 $$H_X(m)\leq\triangle\begin{pmatrix}
 m+n\\n
 \end{pmatrix}.$$
 \end{lemma}
\begin{lemma}[Nesterenko \cite {Nes}, Sombra \cite {Som}, Theorem 4]\label{l:22}  Let X be a projective subvariety of $\mathbb{P}^M$ defined over K of dimension n and degree $\triangle$. Then
$$H_X(m)\geq\begin{pmatrix}m+n+1\\n+1
\end{pmatrix}-\begin{pmatrix}
m-\triangle+n+1\\n+1
\end{pmatrix}$$
for $m\geq 1.$
\end{lemma}
We will derive the following result from Lemma \ref{l:21} and Lemma \ref{l:22}.
\begin{proposition}\label{p:23}Let X be a projective subvariety of $\mathbb{P}^M$ defined over K of dimension n and degree $\triangle$. Let $d\in\mathbb{N}^*$ be a given constant. Then, for any given $\epsilon\geq 0,$ there exists an effectively computable constant $a_{\epsilon}$ depending only on $\epsilon, n, \triangle, d$ such that:
$$\dfrac{m(H_X(m)+1)}{\sum_{i=1}^{m/d-1}H_X(id)}\leq d(n+1+\epsilon),$$
for all $ m\in\mathbb{N}, m\geq a_{\epsilon},d|m.$
\end{proposition}
To prove the above proposition, we need a simple lemma from elementary mathematics.
\begin{lemma}\label{l:24}
Let $k, l\in\mathbb{N}^*$, denote by
$$S_k(l):=1^k+2^k+\cdots+l^k.$$
Then 
$$\dfrac{(l+1)^{k+1}}{k+1}\geq S_k(l)\geq \dfrac{(l+1)^{k+1}}{k+1}-\dfrac{(l+1)^k}{2}.$$
\end{lemma}
\begin{proof} For $k=1$, it is obvious. Thus, we can assume that $k>1$.
We use Newton's polynomial to prove the left hand-side inequality. We have:
$$(m+1)^{k+1}-m^{k+1}=(k+1)m^k+\binom{k+1}{2}m^{k-1}+\cdots+(k+1)m+1.$$
Hence,
\begin{align*}
\sum_{m=1}^l\left((m+1)^{k+1}-m^{k+1}\right)\\
=\sum_{m=1}^l\left(
(k+1)m^k+\binom{k+1}{2}
m^{k-1}+\cdots+(k+1)
m+1\right).
\end{align*}

Therefore, we have
\begin{align}\label{e:21}(l+1)^{k+1}-1=(k+1)S_k(l)+\binom{k+1}{2}
S_{k-1}(l)
+\cdots+(k+1)S_1(l)+l,
\end{align}
which implies the left hand-side inequality.

We replace $k+1$ by $k$ in \eqref{e:21}
\begin{align}\label{e:22}
(l+1)^k-1=kS_{k-1}(l)+\binom{k}{2}
S_{k-2}(l)+\cdots+kS_1(l)+l.
\end{align}
It is easy to see that for $k\geq h\geq 2$, we have
$$\binom{k+1}{h}
\leq \dfrac{k+1}{2}\binom {k}{h-1}
.$$
Combining with \eqref{e:22}, we have
\begin{align}\label{e:23}&\binom{k+1}{2}S_{k-1}(l)+\cdots+(k+1)S_1(l)+l\\
&\leq\dfrac{k+1}{2}\left(kS_{k-1}(l)+\binom{k}{2}
S_{k-2}(l)+\cdots+kS_1(l)+l\right)=\dfrac{k+1}{2}((l+1)^k-1)\notag.
\end{align}
From \eqref{e:21} and \eqref{e:23}, we have
$$(l+1)^{k+1}-1\leq(k+1)S_k(l)+\dfrac{k+1}{2}((l+1)^k-1),$$
which implies the right hand-side inequality.
\end{proof}

We now prove Proposition \ref{p:23}.
 \begin{proof}
 Let $n,\triangle\in\mathbb{N}^*$ be given constants, for each $z\in\mathbb{N}, z\geq 1$, we denote by
$$G(z):=\binom
{z+n+1}{n+1}-\binom{
z-\triangle+n+1}{n+1}.$$
Then, $G(z)$ is a polynomial with leading coefficient $\dfrac{\triangle}{n!},$
$$G(z)=\dfrac{\triangle z^{n}}{n!}+c_1z^{n-1}+\cdots+c_{n},$$
where $c_i, 1\leq i\leq n$ are constants depending on $\triangle,n.$ 

For each $t\in\mathbb{N}^*,$ denote by $$T(t):=\sum_{i=1}^tG(id).$$
Then we have
$$T(t)=\dfrac{\triangle d^n}{n!}S_{n}(t)+c_1d^{n-1}S_{n-1}(t)+\cdots+ c_{n-1}dS_1(t)+tc_{n}.$$
Applying Lemma \ref{l:24}, we have 
\begin{align*}T(t)\geq \dfrac{\triangle d^n}{n!}\left( \dfrac{(t+1)^{n+1}}{n+1}-\dfrac{(t+1)^{n}}{2}\right)
-\dfrac{d^{n-1}|c_1|}{n}(t+1)^n\\
-\cdots-\dfrac{d|c_{n-1}|}{2}(t+1)^2-|c_n|t\\
=\dfrac{\triangle d^n}{(n+1)!}(t+1)^{n+1}-\left(\dfrac{\triangle d^n}{2n!}+\dfrac{d^{n-1}|c_1|}{n}\right)(t+1)^n\\
-\dfrac{d^{n-2}|c_2|}{n-1}(t+1)^{n-1}-\cdots- \dfrac{d|c_{n-1}|}{2}(t+1)^2-|c_n|t.
\end{align*}
Let $m\in \mathbb{N}^*, d|m$. Then, we have
\begin{align}\label{e:005}
dT\left(\dfrac{m}{d}-1\right)\geq \dfrac{\triangle}{(n+1)!}m^{n+1}-\left(\dfrac{\triangle d}{2n!}+\dfrac{|c_1|}{n}\right)m^n\\
\qquad\qquad\qquad\qquad-\cdots- \dfrac{|c_{n-1}|}{2}m^2-|c_n|(m-d).\notag
\end{align}
Now, we estimate $c_i, 1\leq i\leq n.$ Set $$f(z)=(z+n+1)\ldots (z+1)=z^{n+1}+a_1z^n+\ldots+ a_nz+a_{n+1}.$$
Thus, we have
\begin{align*}
a_1=1+\ldots+(n+1)=\dfrac{(n+1)(n+2)}{2},\\
1+a_1+\ldots+a_{n+1}=(n+2)!
\end{align*}
Since $a_i>0, 1\leq i\leq n+1,$ then we have $a_i<(n+2)!.$ Since $G(z)=\frac{f(z)-f(z-\triangle)}{(n+1)!}$, we have
$$c_k=\dfrac{f^{(n-k)}(0)-f^{(n-k)}(-\triangle)}{(n+1)!(n-k)!}, 1\leq k\leq n.$$
Therefore,
\begin{align*}
c_1=-\dfrac{1}{(n+1)!}\left(\binom{n+1}{n-1}(-\triangle)^2+a_1\binom{n}{n-1}(-\triangle)\right)
=\dfrac{\triangle (n+2-\triangle)}{2(n-1)!},\\
c_k=-\dfrac{1}{(n+1)!}\left(\binom{n+1}{n-k}(-\triangle)^{k+1}+\sum_{1\leq i\leq k}a_i\binom{n+1-i}{n-k}(-\triangle)^{k+1-i}\right),\\
 (2\leq k\leq n).
\end{align*}
Therefore, for all $2\leq k\leq n$, we have
\begin{align*}
|c_k|\leq \dfrac{1}{(n+1)!} (k+1)\triangle^{k+1}\cdot\max_ia_i\cdot\max_{1\leq i<j\leq n}\binom{i}{j}
\leq (n+1)^2(2\triangle)^{n+1}.
\end{align*}
Combining with \eqref{e:005}, we have
\begin{align}\label{e:24}
dT\left(\dfrac{m}{d}-1\right)\geq \dfrac{\triangle}{(n+1)!}m^{n+1}-\dfrac{\triangle d+\triangle |n+2-\triangle|}{2n!}m^n\\
\qquad\qquad\qquad\qquad-\cdots- (n+1)^2(2\triangle)^{n+1}(m^{n-1}+\ldots+m^2+m-d)\notag\\
\geq \dfrac{\triangle}{(n+1)!}m^{n+1}-\dfrac{\triangle d+\triangle |n+2-\triangle|}{2n!}m^n- (n+1)^3(2\triangle)^{n+1}m^{n-1}.\notag
\end{align}

 It follows from Lemma \ref{l:22} that
\begin{align}\label{e:25}\sum_{i=1}^{m/d-1}H_X(id)\geq \sum_{i=1}^{m/d-1}G(id)=T(m/d-1).\end{align}
Lemma \ref{l:21} implies that
\begin{align}\label{e:26}H_X(m)\leq\triangle\binom{m+n}{n}
< \triangle\dfrac{(m+n)^n}{n!}.\end{align}

Choose a constant $c>0$ such that the right  hand-side of \eqref{e:24} takes positive value for all $m>c$. Then, from \eqref{e:26}, \eqref{e:25}, \eqref{e:24}, for such $m$ we have
\begin{align*}\dfrac{m(H_X(m)+1)}{d\sum\limits_{i=1}^{m/d-1}H_X(id)}&\leq
\dfrac{\frac{\triangle. m(m+n)^n}{n!}}{ \frac{\triangle}{(n+1)!}m^{n+1}-\frac{\triangle d+\triangle |n+2-\triangle|}{2n!}m^n- (n+1)^3(2\triangle)^{n+1}m^{n-1}}
\end{align*}
From the above inequality, it is easy to see that for each given $\epsilon>0,$ there exists $a_\epsilon$ satisfying Proposition \ref{p:23}.
\end{proof}

We recall a simple lemma by Masser and Wustholz from \cite{MW} on the solutions of a system of linear equations over $K$ which is modified by Ru and Wang (\cite{RW}, Lemma 11).

For positive integers $p, q$ and reals $\nu_\mathfrak{p}$ for all $\mathfrak{p}\in M_K$, we consider the system
\begin{align}\label{e:005}
a_{j1}x_1+\cdots+a_{jp}x_p=0, 1\leq j\leq q 
\end{align}
where $a_{ij}\in K$ not all zero $(1\leq i\leq p, 1\leq j \leq q)$ and $\ord_\mathfrak{p}(a_{ij})\geq \nu_\mathfrak{p}$ for each $p\in M_K.$
\begin{lemma}\label{l:35}
For an integer $t$ with $1\leq t\leq p$ suppose that the system \eqref{e:005} has a solution with $x_1,\ldots, x_p\in K$ such that $x_t\not=0.$ Then, there exists a positive integer $l\leq p-1$ such that the system \eqref{e:005} has a solution $x_1,\ldots,x_p\in K$ with $x_t\not=0$ and $\ord_\mathfrak{p}(x_i)\geq l \nu_\mathfrak{p}$ for each $1\leq i\leq p$ and each $\mathfrak{p}\in M_K. $
\end{lemma}
We also recall the following theorem due to Hermann \cite{He}, Seidenberg \cite{Se}, and Renschuch \cite{Re}. We refer to Aschenbrenner \cite{As} for more discussion.
\begin{theorem}
Let $P_1,\ldots, P_l\in K[X_1,\ldots,X_M]$ be polynomials of total degree at most $d$. If $Q$ is in the ideal generated by $P_1,\ldots,P_l$, then
$$Q=A_1P_1+\ldots+ A_lP_l$$
for certain $A_1,\ldots,A_l\in K[X_1,\ldots,X_M]$ whose degrees are bounded by $(2d)^{2^M}.$
\end{theorem}

Now, we recall the following version of an effective Hilbert's Nullstellensatz (See  \cite{Je},\cite{Ko}, also \cite{RW}, Theorem 12).
\begin{theorem}{\label{theoremB}}\,Let $P_0,\ldots, P_l$ be homogeneous polynomials in $K[X_0,\ldots,X_M]$ of total degree at most $d$ such that $P_0$ vanishes at all common zeros (if any) of $P_1,\ldots,P_l$ in $\bar{K}^M$. Then there exist a positive integer $u\leq (4d)^{M+2}$ and homogeneous polynomials $A_1,\ldots,A_l$ in $K[X_0,\ldots,X_M]$ of total degree at most $(4d)^{M+2}$, such that
$$aP_0^u=A_1P_1+\cdots+A_lP_l$$
for some non-zero element a of $ K.$ Furthermore, there exists a positive integer 
$$l_0\leq l(4(4d)^{M+2})^M$$
such that
$$\min\{\ord_\mathfrak{p}(\alpha),e_\mathfrak{p}(A_1),\ldots,e_\mathfrak{p}(A_l)\}\geq l_0\cdot
\min_{0\leq i\leq l}\{e_\mathfrak{p}(P_i)\}$$
for each $\mathfrak{p}\in M_K.$
\end{theorem}

 Let $X$ be a smooth $n$-dimensional irreducible projective subvariety of  $\mathbb{P}^M$ defined over $K$ of degree $\triangle$.  Let $I_X$ be the prime ideal of  $K[X_0,\ldots, X_M]$ defining $X$. For each integer $m$, let $K[X_0,\ldots, X_M]_m$ denote the vector space of homogeneous polynomials of degree $m$ in $K[X_0,\ldots, X_M]$ (including 0). Denote by $(I_X)_m=K[X_0,\ldots, X_M]_m\cap I_X.$
 Let $P_1,\ldots, P_r\in K[X_0,\ldots, X_M]$ be the canonical polynomials from the Chow form $F_X$ of $X$ defined in \eqref{e:002}. 
By using the same method as in Ru-Wang \cite{RW}, Lemma 14, we will prove a slight generalization of this result. The proofs are almost the same. We only modify some constants apprearing in the proofs.  
  
 \begin{lemma} {\label{l:3.6}} Let $X\subset \mathbb{P}^M$ be a smooth projective variety defined over $K$ with dimension $n\geq 1$ and degree $\triangle$ and $m\geq\max\{3,(n+1)\triangle\}$. Let $\phi_1,\ldots,\phi_{H_X(m)}$ be a fixed  monomial basis of $K[X_0,\ldots,X_M]_m/(I_X)_m$. 
Let $Q$ be a homogeneous polynomial of degree $m$ satisfying $e_\mathfrak{p}(Q)\leq 0$ for all $\mathfrak{p}\in M_K$. Then, there exist $\alpha_0\not =0$ and $\alpha_j, (1\leq j\leq H_X(m))$ in $K$ such that
$$\alpha_0Q\equiv\sum_{j=1}^{H_X(m)}\alpha_j\phi_j\;\,\text{mod}\,\; I_X.$$
Furthermore, 
  we have
$$\ord_\mathfrak{p}(\alpha_j)\geq b(m,n,M)(e_\mathfrak{p}(F_X)+e_\mathfrak{p}(Q)), (0\leq j\leq H_X(m))$$
and 
$$\ord_\mathfrak{p}(\alpha_j)\deg\mathfrak{p}\leq -b(m,n,M)\sum_{\mathfrak{q}\in M_K\backslash\{\mathfrak{p}\}}(e_\mathfrak{q}(F_X)+e_\mathfrak{q}(Q))\deg \mathfrak{q}\;\text{if} \;\alpha_j\not=0$$
where $$b(m,n,M)=(4m)^{n+1}+(5(n+1)\triangle)^{\frac{(n+1)M(M-1)}{2}+M2^M}.$$
\end{lemma}

\begin{proof} As $h(F_X)=h(\alpha F_X)$ for $\alpha\in K^*$, we may assume that one of coefficients of $F_X$
is 1. Without loss of generality, we also assume that $X$ is not contained in the coordinate hyperplane $\{X_0=0\}$ of $\mathbb{P}^M$. Since $\phi_1,\ldots,\phi_{H_X(m)}$ is a fixed monomial basis of $K[X_0,\ldots,X_M]_m/ (I_X)_m$, there exist $\gamma_i\in K, (1\leq i\leq H_X(m))$ such that
$$Q-\sum_{j=1}^{H_X(m)}\gamma_j\phi_j\in I_X.$$
Put $G=Q-\sum_{j=1}^{H_X(m)}\gamma_j\phi_j.$ Let $p_1,\ldots,p_r$ and $g$ be the dehomogenisation of $P_1,\ldots,P_r$ and $G$, respectively, along $X_0\not=0.$ Then, by theorem \ref{theoremA} and \ref{theoremB}, there exist $g_1,\ldots,g_r\in K[X_1,\ldots,X_M]$ with degree bounded by
$$(2(n+1)\triangle)^{2^M}$$
(here we note that the degree of $p_i$ is at most $(n+1)\triangle $ )
such that
$$g=g_1p_1+\cdots g_rp_r.$$
We then homogenize the above equation to obtain
$$X_0^uG=G_1P_1+\cdots G_rP_r,$$
where $u\leq (2(n+1)\triangle)^{2^M}.$ Since $G$ and $P_1,\ldots,P_r$ are homogeneous, we may further assume that $G_1,\ldots,G_r$ are homogeneous. We take $\alpha_0=1, \alpha_j=\gamma_j (1\leq j\leq H_X(m))$ and obtain
\begin{align}\label{e:007}
X_0^u(\alpha_0 Q-\sum_{j=1}^{H_X(m)}\alpha_j\phi_j)=G_1P_1+\cdots G_rP_r \in I_X,
\end{align}
Comparing the monomials in $X_0,\ldots,X_M$ in \eqref{e:007}, we obtain a system of linear equations in the coefficients of $G_i(X_0,\ldots,X_M) (1\leq i\leq r)$ and $\alpha_j (0\leq j\leq H_X(m))$. Note that $G_i(X_0,\ldots,X_M)$ is a homogeneous polynomial in $M+1$ variables and of total degree no bigger than $(2(n+1)\triangle)^{2^M}$, so the number of the coefficients of $G_i$ is at most $\binom{(2(n+1)\triangle)^{2^M}+M}{M}.$ Therefore, the total number of unknowns of this linear system is 
\begin{align}\label{e:008}
p\leq r\cdot\binom{(2(n+1)\triangle)^{2^M}+M}{M}+H_X(m)+1
\end{align}
Applying Lemma \ref{l:35}  to this linear system with $\alpha_0\not=0$, we may select new coefficients of $G_i(X_0,\ldots,X_M),  (1\leq i\leq r)$ and $\alpha_j$ such that $\alpha_0\not=0.$ Furthermore, there exists a positive integer $l\leq p-1$ such that
\begin{align}\label{e:009}
\ord_\mathfrak{p}(\alpha_j)\geq l\cdot \min\{e_\mathfrak{p}(F_X), e_\mathfrak{p}(Q)\}\geq l\cdot (e_\mathfrak{p}(F_X)+ e_\mathfrak{p}(Q)),\\
 \mathfrak{p}\in M_K, 0\leq j\leq H_X(m).\notag
\end{align}
(Note that $e_\mathfrak{p}(Q)\leq 0$.)

If $\alpha_j\not=0,$ by sum formula, we have
$$\ord_\mathfrak{p}(\alpha_j)\deg\mathfrak{p}=-\sum_{\mathfrak{q}\in M_K\backslash\{\mathfrak{p}\}}\ord_\mathfrak{q}(\alpha_j)\deg \mathfrak{q}\leq -l\sum_{\mathfrak{q}\in M_K\backslash\{\mathfrak{p}\}}(e_\mathfrak{q}(F_X)+e_\mathfrak{q}(Q))\deg \mathfrak{q}$$
Moreover, since $I_X$ is a prime ideal and $X$ is not contained in the coordinate hyperplane $\{X_0=0\}$, \eqref{e:007} implies that $\alpha_0 Q-\sum_{j=1}^{H_X(m)}\alpha_j\phi_j\in I_X$, where $\alpha_j'$s are the new coefficients satisfying \eqref{e:009}.

Now, we estimate $p$ introduced in \eqref{e:008}. We have
\begin{align*}H_X(m)\leq \triangle \binom{m+n}{n}&=\triangle\dfrac{(m+n)\cdots (m+1)}{n!}\\
&< m(m+n)^n\leq (2m)^{n+1}\; (\text{notice that}\; m>(n+1)\triangle)
\end{align*}
In a similar way, we have
$$\binom{(2(n+1)\triangle)^{2^M}+M}{M}\leq  (2(2(n+1)\triangle)^{2^M})^M\leq(4(n+1)\triangle)^{2^MM}.$$
The number $r$ introduced in \eqref{e:004} can be bounded by
\begin{align}\label{e:0010}
r\leq\binom{(n+1)\triangle+\frac{M(M-1)}{2}}{(n+1)\triangle}^{n+1}\leq (5(n+1)\triangle)^{\frac{(n+1)M(M-1)}{2}}
\end{align}
Here, we use the following inequality
\begin{align*}
\binom{A+B}{B}&\leq \dfrac{(A+B)^{A+B}}{A^AB^B}=\left(1+\dfrac{B}{A}\right)^A\cdot \left(1+\dfrac{A}{B}\right)^B\\
&\leq\left(e\left(1+\dfrac{A}{B}\right)\right)^B=e^B\left(\dfrac{1}{A}+\dfrac{1}{B}\right)^B\cdot A^B,
\end{align*}
where $A, B$ are positive integers and $e$ is the natural exponential number.
Hence,
$$l\leq p-1\leq (4m)^{n+1}+(5(n+1)\triangle)^{\frac{(n+1)M(M-1)}{2}+2^MM}$$

\end{proof}

We recall a Lemma from (\cite{RW}, Lemma 15)

\begin{lemma}\label{lemmaC} Let the notation be as in Lemma \ref{l:3.6}. We define 
$$\Phi(x)=[\phi_1(x):\cdots:\phi_{H_X(m)}(x)].$$
 Then for each $\mathfrak{p}\in M_K$ and every $x=(x_0,\ldots,x_M)\in X(K)$, we have
$$me_{\mathfrak{p}}(x)\deg \mathfrak{p}\leq e_{\mathfrak{p}}(\Phi(x))\deg \mathfrak{p}\leq me_{\mathfrak{p}}(x)\deg \mathfrak{p}+b(m,n,M)h(F_X),$$
and
$$m h(x)-(M+2)b(m,n,M)h(F_X)\leq h(\Phi(x))\leq m h(x).$$
\end{lemma}

By using the same method as in Ru-Wang \cite{RW}, Lemma 16, we will prove a slight generalization of this result from general position to sub-general position. The proofs are almost the same. We only modify some constants apprearing in the proofs.

\begin{lemma}{\label{l:37} Let the notation be as in Lemma \ref{l:3.6}.  Let $Q_1,\ldots,Q_q,  (q\geq M+1)$ be homogeneous polynomials in $K[X_0,\ldots,X_M]$ of degree $d$, in $N$-subgeneral position with respect to X. For  given $\mathfrak{p}\in M_K$, and $  x\in X\backslash\cup_{i=1}^q\{Q_i=0\},$ we assume that 
$$\ord_{\mathfrak{p}}(Q_1(x))\geq\cdots\geq \ord_{\mathfrak{p}}(Q_q(x)).$$
Then
\begin{align*}&\ord_{\mathfrak{p}}(Q_i(x))\deg \mathfrak{p}-d\cdot e_\mathfrak{p}(x)\deg \mathfrak{p}\\
&\leq\left(6\max\{(N+1)\triangle, d\}\right)^{(n+1)(M^2+M)}\left(h(F_X)+h(Q_1,\ldots,Q_q)\right).
\end{align*}
for $\mathfrak{p}\in M_K$ and $N+1\leq i\leq q.$}
\end{lemma}
\begin{proof}
As $h(F_X)=h(\alpha F_X)$ for $\alpha\in K^*$, we may assume that one of coefficients of $F_X$
is 1. Similarly, since $h(Q_1,\ldots,Q_q)=h(\alpha Q_1,\ldots,\alpha Q_q)$, we can make the same assumption for $Q_1$. Consequently, we have
$$e_\mathfrak{p}(F_X)\leq 0, \min_{1\leq i\leq q}e_\mathfrak{p}(Q_i)\leq 0,$$
for each $\mathfrak{p}\in M_K.$  Let 
$$d'=\max\{\deg P_1,\ldots, \deg P_r, d\}.$$

Since $Q_1, \ldots,Q_q$ are in $N-$subgeneral position with respect to $X\subset\mathbb{P}^M,$ then $P_1,\ldots,P_r,$
$ Q_1,\ldots, Q_{N+1}$ have no zeros in $\mathbb{P}^M(\bar{K})$. Theorem \ref{theoremB} tell us that there exists a constant $u\leq (4d')^{M+2}$ and polynomials $A_{j1},\ldots, A_{jr}, A_{j,r+1}, \ldots,A_{j,r+N+1}\in K[X_0,\ldots,X_M] $ of total degree at most $(4d')^{M+2}$ such that for $0\leq j\leq M$, we have
$$\alpha_jX_j^u=A_{j1}P_1+\cdots+A_{jr}P_r+A_{j,r+1}Q_1+\cdots+A_{j,r+N+1}Q_{N+1}$$
for some non-zero elements $\alpha_j$ of $K.$ Furthermore, there exists a positive integer
\begin{align}\label{e:011}  l_0\leq (r+N+1)(4(4d')^{M+2})^M
\end{align}
such that
\begin{align}\label{e:012}\min\{\ord_\mathfrak{p}(\alpha_0),\ldots,\ord_\mathfrak{p}(\alpha_M),e_\mathfrak{p}(A_{j1}),\ldots,e_\mathfrak{p}(A_{j,r+N+1})\}\\
\geq l_0\min\{e_\mathfrak{p}(P_i), e_\mathfrak{p}(Q_i)\}
\geq l_0\cdot\left(e_\mathfrak{p}(F_X)+\min_{1\leq i\leq q}e_\mathfrak{p}(Q_i)\right)\notag
\end{align}
for each $\mathfrak{p}\in M_K.$

 We may assume that $A_{ji}, (1\leq i\leq r+N+1)$ are homogeneous polynomials and therefore the degrees of $A_{j,r+1},\ldots,A_{j,r+N+1}$ are $u-d$.

Let $x\in X(K)\backslash\cup^q_{i=1}\{Q_i=0\}$. Then
$$\alpha_jx_j^u=A_{j,r+1}(x)Q_1(x)+\cdots+A_{j,r+N+1}(x)Q_{N+1}(x),$$
and hence, for all $j$, we have
\begin{align*}\ord_{\mathfrak{p}}(\alpha_j)+u. \ord_{\mathfrak{p}}(x_j)=\ord_{\mathfrak{p}}(A_{j,r+1}(x)Q_1(x)+\cdots+A_{j,r+N+1}(x)Q_{N+1}(x))\\
\geq \min_{1\leq i\leq N+1} \ord_{\mathfrak{p}}(A_{j,r+i}(x)Q_i(x))\\
\geq  \min_{1\leq i\leq N+1} \ord_{\mathfrak{p}}(A_{j,r+i}(x))+\min_{1\leq i\leq N+1} \ord_{\mathfrak{p}}(Q_i(x))\\
\geq (u-d)e_{\mathfrak{p}}(x)+\ord_{\mathfrak{p}}(Q_{N+1}(x))+ l_0\cdot\left(e_\mathfrak{p}(F_X)+\min_{1\leq i\leq q}e_\mathfrak{p}(Q_i)\right)
\end{align*} (Here, the last inequality follows from \eqref{e:012}). Hence
\begin{align}\ord_{\mathfrak{p}}(Q_{N+1}(x))\leq d e_{\mathfrak{p}}(x)+\max_{0\leq j\leq M}\{\ord_{\mathfrak{p}}(\alpha_j)\}- l_0\cdot\left(e_\mathfrak{p}(F_X)+\min_{1\leq i\leq q}e_\mathfrak{p}(Q_i)\right)\label{e:a}
\end{align}
for $\mathfrak{p}\in M_K.$
Since $\alpha_j\not=0, (0\leq j\leq M)$, from the sum formula and \eqref{e:012} we have
\begin{align*}\ord_\mathfrak{p}(\alpha_j)\deg \mathfrak{p}=-\sum_{\mathfrak{q}\in M_K\backslash\{\mathfrak{p}\}}\ord_\mathfrak{q}(\alpha_j)\deg \mathfrak{q}\\
\leq\sum_{\mathfrak{q}\in M_K\backslash\{\mathfrak{p}\}}-l_0\cdot\left(e_\mathfrak{q}(F_X)+\min_{1\leq i\leq q}e_\mathfrak{q}(Q_i)\right)\deg \mathfrak{q}.
\end{align*}
 Combining with \eqref{e:a}, we have
 \begin{align*}
  \ord_\mathfrak{p}(Q_{N+1})\deg \mathfrak{p}\leq \sum_{\mathfrak{q}\in M_K\backslash\{\mathfrak{p}\}}-l_0\cdot\left(e_\mathfrak{q}(F_X)+\min_{1\leq i\leq q}e_\mathfrak{q}(Q_i)\right)\deg \mathfrak{q}&\\
 +d\cdot e_\mathfrak{p}(x)\deg \mathfrak{p} -l_0\cdot\left(e_\mathfrak{p}(F_X)+\min_{1\leq i\leq q}e_\mathfrak{p}(Q_i)\right)\deg \mathfrak{p}&\\
 = d\cdot e_\mathfrak{p}(x)\deg \mathfrak{p}+l_0\left(h(F_X)+h(Q_1,\ldots,Q_q)\right)&
 \end{align*}
Now, we estimate $l_0$ introduced in \eqref{e:011}. By \eqref{e:0010}, we have 
$$l_0\leq \left(6\max\{(N+1)\triangle, d\}\right)^{(n+1)(M^2+M)}$$
\end{proof}

\section{Proof of  Theorem \ref{t:11}} 

 Now, we recall the following effective version of the classical Schmidt's subspace theorem for function fields in \cite{W}. 
\vskip0.2cm
\noindent
\textbf{Theorem A.}
 {\it Let $K$ be the function field of a nonsingular projective variety $V$ defined over an algebraically closed field of characteristic 0. Let $H_1,\ldots,H_q$ be hyperplanes in $\mathbb{P}^n(K)$ and S be a finite set of prime divisors of  V. Then there exists an effectively computable  finite union $\mathcal{R}$ of proper linear subspaces of $\mathbb{P}^n(K)$ depending only on the given hyperplanes such that the following is true. Given $\epsilon>0$, there exist effectively computable constants $c_{\epsilon}$ and $c'_{\epsilon}$ such that for any $x\in\mathbb{P}^n(K)\backslash\mathcal{R}$ either
$$h(x)\leq c_{\epsilon},$$
or $$\sum_{\mathfrak{p}\in S}\max_J\sum_{j\in J}\lambda_{\mathfrak{p},H_j}(x)\leq (n+1+\epsilon)h(x)+c'_{\epsilon},$$
where the maximum is taken over all subsets J of $\{1,\ldots,q\}$ such that the linear forms $H_j, j\in J$ are linearly independent.}\\
\begin{remark}\label{re:1} The constants $c_{\epsilon}$ and $c'_{\epsilon}$ depend on $\epsilon$, the degree of the canonical divisor class of $V$, the projective degree of $V$, the degree of $S$ (i.e $\sum_{\mathfrak{p}\in S} \deg \mathfrak{p}$) and $h(H_1,\ldots, H_q).$
\end{remark}
We will prove  that theorem \ref{t:11} is an implication of the following theorem.
\begin{theorem}\label{t:41}
Let $K$ be the function field of a nonsingular projective variety $V$ defined over an algebraically closed field of characteristic 0. Let $X$ be a smooth n-dimensional projective subvariety of $\mathbb{P}^M$ defined over K with projective degree $\triangle$. Let S be a finite set of prime divisors of V. Let $Q_i, 1\leq i\leq q,$ be homogeneous polynomials  in $K[X_0,\ldots,X_M]$ in N-subgeneral position with respect to X.

Assume further that, $Q_i, 1\leq i\leq q$ have the same degree $d$ and for each $i$, $Q_i$ has at least one coefficient equal to 1.  Then for any given $\epsilon>0$, there exists an effectively computable finite union $\mathfrak{U}_{\epsilon}$ of proper algebraic subsets of $\mathbb{P}^M(K)$ not containing $X$  and effectively computable constants $\tilde{c}_{\epsilon}, \tilde{c}'_{\epsilon}$ such that for any $x\in X \backslash\mathfrak{U}_{\epsilon}$ either
$$h(x)\leq \tilde{c}_{\epsilon},$$
or 
$$\sum_{i=1}^q\sum_{\mathfrak{p}\in S}d^{-1}\lambda_{\mathfrak{p},Q_i}(x)\leq (N(n+1)+\epsilon)h(x)+\tilde{c}'_{\epsilon}.$$

 The constants $\tilde{c}_{\epsilon}, \tilde{c}'_{\epsilon}$ given in \eqref{e:014}, \eqref{e:015} and the algebraic subsets in $\mathfrak{U}_\epsilon$ depend on $\epsilon, M,N, q, K, S, V, X$ and the $Q_i.$
\end{theorem}
\begin{proof}

We will fix a $\mathfrak{p}\in S$ first. For each $x=[x_0,\ldots,x_M]\in \mathbb{P}^M(K)\backslash\cup^q_{i=1}\{Q_i=0\},$ there exists a renumbering $l_1(\mathfrak{p},x),\ldots,l_q(\mathfrak{p},x)$ of the indices $1,\ldots,q$ such that 
$$\ord_{\mathfrak{p}}(Q_{l_1({\mathfrak{p},x}}))\geq \ord_{\mathfrak{p}}(Q_{l_2({\mathfrak{p},x}}))\geq\cdots\geq \ord_{\mathfrak{p}}(Q_{l_q({\mathfrak{p},x}})).$$
Since $Q_1,\ldots, Q_q$ are in $N-$subgeneral position with respect to $X$, it follows from Lemma \ref{l:37} that 
\begin{align}\label{e:31}\sum^{q}_{i=1}\lambda_{\mathfrak{p},Q_i}(x)=\sum^{q}_{i=1}\left(\ord_{\mathfrak{p}}(Q_i(x))-de_{\mathfrak{p}}(x)-e_{\mathfrak{p}}(Q_i)\right)\text{deg} \mathfrak{p}\\
\leq\sum_{i=1}^N\left(\ord_{\mathfrak{p}}(Q_{l_i(\mathfrak{p},x)}(x))-de_{\mathfrak{p}}(x)\right)\deg \mathfrak{p}
+a(q-N)-\sum_{i=1}^q e_{\mathfrak{p}}(Q_i) \deg\,\mathfrak{p},\notag\\
\leq N\left(\ord_{\mathfrak{p}}(Q_{l_1(\mathfrak{p},x)}(x))-de_{\mathfrak{p}}(x)\right)\text{deg} \mathfrak{p}+a(q-N)-q e_{\mathfrak{p}}(Q_1,\ldots,Q_q) \deg\mathfrak{p},\notag
\end{align}
where $a=\left(6\max\{(N+1)\triangle, d\}\right)^{(n+1)(M^2+M)}(h(F_X)+h(Q_1,\ldots,Q_q)).$

For every positive integer m with $d|m$, we consider the following filtration on the vector space $K[X_0,\ldots,X_M]_m/(I_X)_m$ with respect to $Q_{l_1(\mathfrak{p},x)}$
$$X_m=W_0^{l_1(\mathfrak{p},x)}\supset W_1^{l_1(\mathfrak{p},x)}\supset\cdots\supset W_{m/d}^{l_1(\mathfrak{p},x)}$$
is defined by
$$W_i^{l_1(\mathfrak{p},x)}=\{g^*|g\in K[X_0,\ldots,X_M]_m, Q_{l_1(\mathfrak{p},x)}^i|g \},$$
where $g^*$ is the projection of $g$ to $K[X_0,\ldots,X_M]_m/(I_X)_m.$ Take a basis $\psi_1^{l_1(\mathfrak{p},x)},\ldots,\psi_{H_X(m)}^{l_1(\mathfrak{p},x)}$ of the vector space $K[X_0,\ldots,X_M]_m/(I_X)_m$ compatible with the filtration $W_i^{l_1(\mathfrak{p},x)},$ by this we mean that, for each $i=1,\ldots,m/d,$ it contains a basis of $W_i^{l_1(\mathfrak{p},x)}.$ 
 We can choose a basis $\psi_1^{l_1(\mathfrak{p},x)},\ldots,\psi_{H_X(m)}^{l_1(\mathfrak{p},x)}$ such that they can be written as following
\begin{align}\label{e:32}
\psi_j^{l_1(\mathfrak{p},x)}=Q_{l_1(\mathfrak{p},x)}^{i_j(\mathfrak{p},x)}.g_j,\end{align}
where $g_j\in X_{m-di_j(\mathfrak{p},x)}$ is chosen to be a monomial and $Q_{l_1(\mathfrak{p},x)}$ does not divide $g_j$. 
Hence
\begin{align*}\ord_{\mathfrak{p}}(\psi_j^{l_1(\mathfrak{p},x)}(x))&=i_j(\mathfrak{p},x)\ord_{\mathfrak{p}}(Q_{l_1(\mathfrak{p},x)}(x))+\ord_{\mathfrak{p}}(g_j(x))\\
&\geq i_j(\mathfrak{p},x)\ord_{\mathfrak{p}}(Q_{l_1(\mathfrak{p},x)}(x))+(m-di_j(\mathfrak{p},x))e_{\mathfrak{p}}(x)
\end{align*}
Therefore
\begin{align}\label{e:33}\sum_{j=1}^{H_X(m)}\ord_{\mathfrak{p}}(\psi_j^{l_1(\mathfrak{p},x)}(x))\geq\left(\sum_{j=1}^{H_X(m)} i_j(\mathfrak{p},x)\right)\ord_{\mathfrak{p}}(Q_{l_1(\mathfrak{p},x)}(x))\\
+\left(mH_X(m)
-d\left(\sum_{j=1}^{H_X(m)}i_j(\mathfrak{p},x)\right)\right)e_{\mathfrak{p}}(x).\notag
\end{align}

Now, we estimate $\sum_{j=1}^{H_X(m)} i_j(\mathfrak{p},x).$

It is clear that there are exactly $\dim(W_i^{l_1(\mathfrak{p},x)}/W_{i+1}^{l_1(\mathfrak{p},x)})$ elements $\psi_j^{l_1(\mathfrak{p},x)}$ with $i_j=i$ in the set $\psi_1^{l_1(\mathfrak{p},x)},\ldots,\psi_{H_X(m)}^{l_1(\mathfrak{p},x)}.$ Hence,
\begin{align}\label{e:34}\sum_{j=1}^{H_X(m)}i_j(\mathfrak{p},x)=\sum_{i=1}^{m/d}i.\dim (W_i^{l_1(\mathfrak{p},x)}/W_{i+1}^{l_1(\mathfrak{p},x)}),
\end{align}
 where $W_{m/d+1}^{l_1(\mathfrak{p},x)}:=\vec{0}.$

We notice that each element $\psi$ of $W_i^{l_1(\mathfrak{p},x)}$ can be represented as $\psi=Q_{l_1(\mathfrak{p},x)}^ig$ with $g\in K[X_0,\ldots, X_M]_{m-id}.$ Furthermore, two polynomials $g_1, g_2$ such that $Q_{l_1(\mathfrak{p},x)}^ig_1=Q_{l_1(\mathfrak{p},x)}^ig_2$ in $W_i^{l_1(\mathfrak{p},x)}$ if and only if $g_1-g_2$ vanishes identically in $X$. Hence, we have $\dim W_i^{l_1(\mathfrak{p},x)}=\dim K[X_0,\ldots,X_M]_{m-id}/(I_X)_{m-id}=H_X(m-id).$ Therefore,
\begin{align}\label{e:35}
\sum_{i=1}^{m/d}i &\dim (W_i^{l_1(\mathfrak{p},x)}/W_{i+1}^{l_1(\mathfrak{p},x)})=\sum_{i=1}^{m/d}i(\dim W_i^{l_1(\mathfrak{p},x)}-\dim W_{i+1}^{l_1(\mathfrak{p},x)})\\
&=\sum_{i=1}^{m/d}i\dim W_i^{l_1(\mathfrak{p},x)}-\sum_{i=1}^{m/d}((i+1)\dim W_{i+1}^{l_1(\mathfrak{p},x)}-\dim W_{i+1}^{l_1(\mathfrak{p},x)})\notag\\
&=\sum_{i=1}^{m/d}\dim W_i^{l_1(\mathfrak{p},x)}=\sum_{i=1}^{m/d}H_X(m-id)=\sum_{i=1}^{m/d-1}H_X(id)\notag
\end{align}
Denote by $S(t):=\sum\limits_{i=1}^{t}H_X(id)$. Then, combining \eqref{e:35}, \eqref{e:34}, \eqref{e:33}, we have
\begin{align*}\sum_{j=1}^{H_X(m)}\ord_{\mathfrak{p}}(\psi_j^{l_1(\mathfrak{p},x)}(x))\geq S(m/d-1)\ord_{\mathfrak{p}}(Q_{l_1(\mathfrak{p},x)}(x))\\
+(mH_X(m)-dS(m/d-1))e_{\mathfrak{p}}(x).\end{align*}
Hence, \begin{align}\label{e:37}
\sum_{j=1}^{H_X(m)}\left(\ord_{\mathfrak{p}}(\psi_j^{l_1(\mathfrak{p},x)}(x))-me_{\mathfrak{p}}(x)\right)\geq S(m/d-1)\left(\ord_{\mathfrak{p}}(Q_{l_1(\mathfrak{p},x)}(x))-de_{\mathfrak{p}}(x)\right)
\end{align}

Let $\phi_1,\ldots,\phi_{H_X(m)}$ be a fixed monomial basis of $K[X_0,\ldots, X_M]_m/ (I_X)_m$. 
Applying Lemma \ref{l:3.6}, for each $i$, there exists a nonzero element $\beta_i$ in $K$ and linear form $L_i\in K[Y_0,\ldots, Y_{H_X(m)-1}]$ with
\begin{align*}b(m,n,M) (e_\mathfrak{p}(F_X)+e_\mathfrak{p}(\psi_i^{l_1(\mathfrak{p},x)})) \deg \mathfrak{p}\leq \ord_\mathfrak{p}(\beta_i)\deg \mathfrak{p}\\
\leq -b(m,n,M)\sum_{\mathfrak{q}\in M_K\backslash \{\mathfrak{p}\}}(e_\mathfrak{q}(F_X)+e_\mathfrak{q}(\psi_i^{l_1(\mathfrak{p},x)}))\deg \mathfrak{q}.
\end{align*}
and
\begin{align*}b(m,n,M) (e_\mathfrak{p}(F_X)+e_\mathfrak{p}(\psi_i^{l_1(\mathfrak{p},x)})) \deg \mathfrak{p}\leq e_\mathfrak{p}(L_i)\deg \mathfrak{p}\\
\leq -b(m,n,M)\sum_{\mathfrak{q}\in M_K\backslash \{\mathfrak{p}\}}(e_\mathfrak{q}(F_X)+e_\mathfrak{q}(\psi_i^{l_1(\mathfrak{p},x)}))\deg \mathfrak{q}.
\end{align*}
such that
$$\beta_i\psi_i^{l_1(\mathfrak{p},x)}\equiv L_i(\phi_1,\ldots,\phi_{H_X(m)})\;\,\text{mod}\,\; I_X, 1\leq i\leq H_X(m)$$
We have
$$e_\mathfrak{p}(\psi_i^{l_1(\mathfrak{p},x)})=e_\mathfrak{p}(Q_{l_1(\mathfrak{p},x)}^{i_j(\mathfrak{p},x)})=i_j(\mathfrak{p},x)e_\mathfrak{p}(Q_{l_1(\mathfrak{p},x)})\geq m e_\mathfrak{p}(Q_1,\ldots,Q_q).$$
(Notice that  since $Q_i$ has one coefficient equal to 1, we have $e_\mathfrak{p}(Q_i)\leq 0$ for all $ \mathfrak{p}\in M_K$)

Therefore, 
\begin{align}\label{e:38}b(m,n,M) (e_\mathfrak{p}(F_X)+me_\mathfrak{p}(Q_1,\ldots,Q_q)) \deg \mathfrak{p}\leq \ord_\mathfrak{p}(\beta_i)\deg \mathfrak{p}\\
\leq -b(m,n,M)\sum_{\mathfrak{q}\in M_K\backslash \{\mathfrak{p}\}}(e_\mathfrak{q}(F_X)+me_\mathfrak{q}(Q_1,\ldots,Q_q))\deg \mathfrak{q}\notag.
\end{align}
and
\begin{align}\label{e:39}b(m,n,M) (e_\mathfrak{p}(F_X)+me_\mathfrak{p}(Q_1,\ldots,Q_q)) \deg \mathfrak{p}\leq e_\mathfrak{p}(L_i)\deg \mathfrak{p}\\
\leq -b(m,n,M)\sum_{\mathfrak{q}\in M_K\backslash \{\mathfrak{p}\}}(e_\mathfrak{q}(F_X)+me_\mathfrak{q}(Q_1,\ldots,Q_q))\deg \mathfrak{q}.\notag
\end{align}
Set $$\Phi(x)=[\phi_1(x):\ldots:\phi_{H_X(m)}(x)]: X\longrightarrow \mathbb{P}^{H_X(m)-1}.$$
Then, 
$$L_i(\Phi(x))=\beta_i\psi_i^{l_1(\mathfrak{p},x)}(x).$$
Combining with $\eqref{e:38}$, we have
\begin{align*}
\ord_\mathfrak{p}(L_i(\Phi(x)))\deg \mathfrak{p}=\ord_\mathfrak{p}(\beta_i)\deg \mathfrak{p}+\ord_\mathfrak{p}(\psi_i^{l_1(\mathfrak{p},x)}(x))\deg \mathfrak{p}\\
\geq \ord_\mathfrak{p}(\psi_i^{l_1(\mathfrak{p},x)}(x))\deg \mathfrak{p}+b(m,n,M) (e_\mathfrak{p}(F_X)+me_\mathfrak{p}(Q_1,\ldots,Q_q)) \deg \mathfrak{p}
\end{align*}
Applying Lemma \ref{lemmaC}, we have
\begin{align*}\lambda_{\mathfrak{p},L_i}(\Phi(x))=\left(\ord_\mathfrak{p}(L_i(\Phi(x))-e_{\mathfrak{p}}(\Phi(x))-e_\mathfrak{p}(L_i)\right)\deg\mathfrak{p}\\
\geq \ord_\mathfrak{p}(\psi_i^{l_1(\mathfrak{p},x)}(x))\deg \mathfrak{p}+b(m,n,M) (e_\mathfrak{p}(F_X)
+me_\mathfrak{p}(Q_1,\ldots,Q_q)) \deg \mathfrak{p}\\
-me_\mathfrak{p}(x)\deg \mathfrak{p}-b(m,n,M)h(F_X)-e_\mathfrak{p}(L_i)\deg\mathfrak{p}
\end{align*}
Combining with \eqref{e:39}, we have
\begin{align}\label{e:3010}
\lambda_{\mathfrak{p},L_i}(\Phi(x))\geq \ord_\mathfrak{p}(\psi_i^{l_1(\mathfrak{p},x)}(x))\deg \mathfrak{p}\\
-me_\mathfrak{p}(x)\deg \mathfrak{p}-b(m,n,M)(2h(F_X)+mh(Q_1,\ldots,Q_q))\notag
\end{align}


 From \eqref{e:37} and \eqref{e:3010}, we have 
\begin{align*}S(m/d-1)\left(\ord_{\mathfrak{p}}(Q_{l_1(\mathfrak{p},x)}(x))-de_{\mathfrak{p}}(x)\right)\deg\mathfrak{p}
\leq  \sum_{j=1}^{H_X(m)}\lambda_{\mathfrak{p},L_j}(\Phi(x))+b_1,
\end{align*}
where $b_1=(m+1) H_X(m)b(m,n,M)\left(h(F_X)+h(Q_1,\ldots,Q_q)\right).$\\
Since there are only $q$ choices of $Q_{l_1(\mathfrak{p},x)}\subset\{Q_1,\ldots,Q_q\},$ we have a finite collection of linear forms $L_1,\ldots,L_u$. Combining the above equation with \eqref{e:31}, we have
\begin{align}\label{e:310}
\sum_{\mathfrak{p}\in S}\sum^{q}_{i=1}\lambda_{\mathfrak{p},Q_i}(x)\leq\dfrac{N}{S(m/d-1)} \sum_{\mathfrak{p}\in S}\sum_{i\in K_{\mathfrak{p}}}\lambda_{\mathfrak{p},L_i}(\Phi(x))+\dfrac{Nb_2}{S(m/d-1)}+b_3,
\end{align}
where $b_2=|S|b_1,$ $b_3=a(q-N)|S|-q\sum_{\mathfrak{p}\in S}e_\mathfrak{p}(Q_1,\ldots,Q_q)\deg \mathfrak{p},$ and $K_\mathfrak{p}$ is an index set
 $\{i_1,\ldots, i_{H_X(m)}\}\subset \{1,\ldots, u\}$ such that $L_{i_1},\ldots, L_{i_{H_X(m)}}$ are linearly independent   and $\sum_{i\in K_\mathfrak{p}}L_{\mathfrak{p}, i}(\Phi(x))$ achieve maximum among all such index sets.

Apply theorem A for the family of linear forms $L_1,\ldots,L_u$ and $\epsilon=1$, then there exists a finite union of effectively computable linear subspaces $\mathcal{R}$ in $\mathbb{P}^{H_X(m)-1}$ and effectively computable constants $c_1, c'_{1}$ such that for all $\Phi(x)$ not contained in $\mathcal{R}$, either
\begin{align*}h(\Phi(x))\leq c_{1},\end{align*}
or 
\begin{align}\label{e:311}\sum_{\mathfrak{p}\in S}\sum_{i\in K_{\mathfrak{p}}}\lambda_{\mathfrak{p},L_i}(\Phi(x))\leq (H_X(m)+1)h(\Phi(x))+c'_{1}\leq(H_X(m)+1)mh(x)+c'_{1}.
\end{align}
In view of remark \ref{re:1} and \eqref{e:39}, the constants $c_{1}, c'_{1}$ can be bounded by the constants depend on $\epsilon$, the degree of the canonical divisor class of $V$, the projective degree of $V$, the degree of $S$ (i.e $\sum_{\mathfrak{p}\in S} \deg \mathfrak{p}$) and  $h(Q_1,\ldots,Q_q), h(F_X).$

By the latter case, equations \eqref{e:310}, \eqref{e:311} yields
\begin{align}\label{e:312}\sum_{\mathfrak{p}\in S}\sum^{q}_{i=1}\lambda_{\mathfrak{p},Q_i}(x)
\leq\dfrac{Nm(H_X(m)+1)}{S(m/d-1)}h(x)+\dfrac{N(b_2+c'_{1})}{S(m/d-1)}+b_3\end{align}
We apply Proposition \ref{p:23} for given $\epsilon>0$, then there exists $a_{\epsilon}$ effectively computable depending on $n, d, \triangle$ such that
$$\dfrac{(H_X(m)+1)m}{S(m/d-1)}<d(n+1+\dfrac{\epsilon}{N}),$$
for all $m>a_\epsilon, d|m.$
We choose $m:=d([a_{\epsilon}/d]+1).$ Then, by the latter case, from \eqref{e:312}, we have
$$\sum_{\mathfrak{p}\in S}\sum_{i=1}^q\lambda_{\mathfrak{p},Q_i}(x)\leq d(N(n+1)+\epsilon)h(x)+b_3+\dfrac{N(b_2+c'_{1})}{S(m/d-1)}.$$
 Moreover, we have
 \begin{align*}b_3+\dfrac{N(b_2+c'_{1})}{S(m/d-1)}=
  \left(6\max\{(N+1)\triangle, d\}\right)^{(n+1)(M^2+M)}\\
 \cdot (h(F_X)+h(Q_1,\ldots,Q_q))\cdot (q-N)|S|
 -q \sum_{\mathfrak{p}\in S}e_\mathfrak{p}(Q_1,\ldots,Q_q)\deg \mathfrak{p}\\
 + \dfrac{N\left(|S|(m+1) H_X(m)b(m,n,M)\left(h(F_X)+h(Q_1,\ldots,Q_q)\right)+c'_{1}\right)}{S(m/d-1)}\leq d\tilde{c}'_\epsilon,
 \end{align*}
 where 
 \begin{align}\label{e:014}
 \tilde{c}'_\epsilon=\frac{1}{d}
  \left(6\max\{(N+1)\triangle, d\}\right)^{(n+1)(M^2+M)}\\
 \cdot (h(F_X)+h(Q_1,\ldots,Q_q))\cdot (q-N)|S|
 +\frac{q}{d} h(Q_1,\ldots,Q_q)\notag\\
 + \dfrac{N\left(|S|(m+1) H_X(m)b(m,n,M)\left(h(F_X)+h(Q_1,\ldots,Q_q)\right)+c'_{1}\right)}{dS(m/d-1)}\notag
 \end{align} 
 (Here, we use the fact that $Q_i$ has at least one coefficient equal to 1, therefore $-\sum_{\mathfrak{p}\in M_K}e_\mathfrak{p}(Q_1,\ldots,Q_q)\deg \mathfrak{p}\leq h(Q_1,\ldots,Q_q)$.
 
  Hence, by the latter case, we have
 \begin{align}\label{e:017}\sum_{\mathfrak{p}\in S}\sum_{i=1}^qd^{-1}\lambda_{\mathfrak{p},Q_i}(x)\leq (N(n+1)+\epsilon)h(x)+\tilde{c}'_\epsilon
 \end{align}
By the first case, apply Lemma \ref{lemmaC}, then we have
\begin{align}\label{e:016}h(x)\leq \tilde{c}_\epsilon,
\end{align}
where 
\begin{align}\label{e:015}
\tilde{c}_\epsilon=\dfrac{c_{1}+(M+2)b(m,n,M)h(F_X)}{m}.
\end{align}

Finally, we may conclude our proof by the following fact. The morphism $\Phi=[\phi_1:\ldots:\phi_{H_X(m)}]: X\longrightarrow \mathbb{P}^{H_X(m)-1}$ has the property that either \eqref{e:016} or \eqref{e:017} hold for $x\in X$ with $\Phi(x)\not\in \mathcal{R}$. We now consider those $x\in X(K)$ with $\Phi(x)\in \mathcal{R}.$ We note that the exceptional linear subspaces $\mathcal{R}$ in theorem A do not depend on $\epsilon$, however in our case they do depend on $m$. Hence we write $\mathcal{R}$ as $\mathcal{R}_m.$ Let $\mathcal{U}_\epsilon$ be the collection of algebraic subsets of $X$ which are the inverse images of the algebraic subsets in $\mathcal{R}_m$ under the morphism $\Phi.$ As the components of the morphism $\Phi(x)$ are monomials of degree $m$ in $X_0,\ldots, X_M$, the conditions that $\Phi(x)$ is contained in a finite union of effectively computable linear subspaces $\mathcal{R}_m$ in $\mathbb{P}^{H_X(m)-1}$ is equivalent to that $x$ is contained in a set $\mathcal{U}_{\epsilon}$ containing finitely many effectively computable algebraic subsets of $\mathbb{P}^M$ with degree no more than $m.$  
\end{proof}
\noindent
\vskip0.2cm
\textbf{Proof of theorem \ref{t:11}}
 
 Let $Q_i, 1\leq i\leq q$ be homogeneous polynomials of degree $d_i$, respectively.  Let $d$ is the l.c.m of $d_i', 1\leq i\leq q.$ For each $i, 1\leq i\leq q,$ we choose one coefficient not equal to 0 of $Q_i$ and denote it by $a_i$.

We apply Theorem \ref{t:41} for $(\frac{Q_i}{a_i})^{d/d_i}, 1\leq i\leq q.$ Then, for a given $\epsilon$, there  there exists an effectively computable finite union $\mathfrak{U}_{\epsilon}$ of proper algebraic subsets of $\mathbb{P}^M(K)$ not containing $X$  such that for any $x\in X \backslash\mathfrak{U}_{\epsilon}$ either
$$h(x)\leq \tilde{c}_{\epsilon},$$
or 
$$\sum_{i=1}^q\sum_{\mathfrak{p}\in S}d^{-1}\lambda_{\mathfrak{p},(\frac{Q_i}{a_i})^{d/d_i}}(x)\leq (N(n+1)+\epsilon)h(x)+\tilde{c}'_{\epsilon},$$
where $\tilde{c}_\epsilon, \tilde{c}'_\epsilon$ are given in \eqref{e:014},\eqref{e:015}. 

The algebraic subsets in $\mathfrak{U}_\epsilon$  depends on $\epsilon, M,N, q, K, S, X$ and the $Q_i.$

Now, we estimate an upper bound for $\tilde{c}'_\epsilon.$

Since $e_\mathfrak{p}((\frac{Q_i}{a_i})^{d/d_i})=\frac{d}{d_i}e_\mathfrak{p}(\frac{Q_i}{a_i})\leq 0,\;\text{for all}\; i=1,\ldots,q.$ We have
$$e_\mathfrak{p}\left(\left(\frac{Q_1}{a_1}\right)^{d/d_1}, \ldots, \left(\frac{Q_q}{a_q}\right)^{d/d_q}\right)\geq e_\mathfrak{p}\left(\left(\frac{Q_1}{a_1}\right)^{d/d_1}\right)+\cdots+e_\mathfrak{p}\left(\left(\frac{Q_q}{a_q}\right)^{d/d_q}\right)$$
\begin{align*}
h\left(\left(\frac{Q_1}{a_1}\right)^{d/d_1}, \ldots, \left(\frac{Q_q}{a_q}\right)^{d/d_q}\right)&\leq h\left(\left(\frac{Q_1}{a_1}\right)^{d/d_1}\right)+\cdots+h\left(\left(\frac{Q_q}{a_q}\right)^{d/d_q}\right)\\
&=\dfrac{d}{d_1}h(Q_1)+\cdots+\dfrac{d}{d_q}h(Q_q)
\end{align*}
Therefore,
 \begin{align}\label{e:314}
 \tilde{c}'_\epsilon
 \leq \left(6\max\{(N+1)\triangle, d\}\right)^{(n+1)(M^2+M)}( \frac{1}{d} h(F_X)+\frac{1}{d_1}h(Q_1)+\ldots+\frac{1}{d_q} h(Q_q))\\
 \cdot (q-N)|S|
 +q(\frac{1}{d_1}  h(Q_1)+\ldots+\frac{1}{d_q} h(Q_q))\notag\\
 + \dfrac{N\left(|S|(m+1) H_X(m)b(m,n,M)\left(h(F_X)+\frac{d}{d_1} h(Q_1)+\ldots+\frac{d}{d_q} h(Q_q)\right)+c'_{1}\right)}{dS(m/d-1)}.\notag
 \end{align}
 We observe that
 \begin{align*}
 \lambda_{\mathfrak{p},(Q_i/a_i)^{d/d_i}}(x)&=\left(\frac{d}{d_i} \ord_{\mathfrak{p}}(Q_i(x)/a_i)-de_{\mathfrak{p}}(x)-\frac{d}{d_i} e_{\mathfrak{p}}(Q_i/a_i)\right)\text{deg}\,\mathfrak{p}\\
 &=\frac{d}{d_i}\lambda_{\mathfrak{p},Q_i}(x)
 \end{align*}
 Therefore,  for any $x\in X \backslash\mathfrak{U}_{\epsilon}$ we have either
$$h(x)\leq c_{\epsilon},$$ where $c_\epsilon=\tilde{c}_\epsilon,$
or 
$$\sum_{i=1}^q\sum_{\mathfrak{p}\in S}d_i^{-1}\lambda_{\mathfrak{p},Q_i}(x)\leq (N(n+1)+\epsilon)h(x)+c'_{\epsilon},$$
where \begin{align}\label{e:315}
 c'_\epsilon
 = \left(6\max\{(N+1)\triangle, d\}\right)^{(n+1)(M^2+M)}(\frac{1}{d} h(F_X)+\frac{1}{d_1}h(Q_1)+\ldots+\frac{1}{d_q} h(Q_q))\\
 \cdot (q-N)|S|
 +q(\frac{1}{d_1}  h(Q_1)+\ldots+\frac{1}{d_q} h(Q_q))\notag\\
 + \dfrac{N\left(|S|(m+1) H_X(m)b(m,n,M)\left(h(F_X)+\frac{d}{d_1} h(Q_1)+\ldots+\frac{d}{d_q} h(Q_q)\right)+c'_{1}\right)}{dS(m/d-1)}.\notag
 \end{align}

This completes the proof.
    

\end{document}